\documentclass{article}
\usepackage[utf8]{inputenc}
\usepackage[english]{babel}
\usepackage[T1]{fontenc}
\usepackage{amsmath}
\usepackage{amsfonts}
\usepackage{amssymb}
\usepackage{amsthm}
\usepackage{makeidx}
\usepackage{graphicx}
\usepackage{tikz}
\usepackage{multicol}
\usepackage[all]{xy}
\usepackage{caption} 
\usepackage{appendix}
\usepackage{subcaption}
\usepackage{titlesec}
\usepackage[margin=1.5in]{geometry}
\usepackage[all]{xy}
\title{Triquadratic $p$-Rational Fields}
\author{Julien Koperecz\\
Laboratoire de Mathématiques de Besançon -- UMR CNRS 6623\\
Université de Franche-Comté\\
16, route de Gray,
25030 Besançon cedex, France}

\titleformat{\chapter}[display]
	{\bfseries \Huge}
	{\filleft\MakeUppercase{\chaptertitlename} \Huge\thechapter}
	{4ex}
	{\titlerule\vspace{2ex} \filright}
	[\vspace{2ex} \titlerule]
	
\titleformat{\section}[hang]
	{\bf \Large \center}
	{\thesection.}
	{1ex}
	{}
	[]  

\if@nosecthm
    \newtheorem{theorem}{Theorem}
\else
    \newtheorem{theorem}{Theorem}[section]
\fi
    \newtheorem{conjecture}[theorem]{Conjecture}
    \newtheorem{corollary}[theorem]{Corollary}
    \newtheorem{lemma}[theorem]{Lemma}
     \newtheorem{proposition}[theorem]{Proposition}
     
    \theoremstyle{definition}
    \newtheorem{definition}[theorem]{Definition}
    \newtheorem{example}[theorem]{Example}

	\theoremstyle{remark}
    \newtheorem{remark}[theorem]{Remark}

\newcommand{\QQ}{\mathbb{Q}}
\newcommand{\ZZ}{\mathbb{Z}}

\begin{document}

\maketitle

\begin{abstract}
In his work about Galois representations, Greenberg conjectured the existence, for any odd prime $p$ and any positive integer $t$, of a multiquadratic $p$-rational number field of degree $2^t$. In this article, we prove that there exists infinitely many  primes $p$ such that the triquadratic field $\QQ(\sqrt{p(p+2)}, \sqrt{p(p-2)},i)$ is $p$-rational.

To do this, we use an analytic result, proved apart in section \S 4, providing us with infinitely many prime numbers $p$ such that $p+2$ et $p-2$ have ``big'' square factors. Therefore the related imaginary quadratic subfields $\QQ(i\sqrt{p+2})$, $\QQ(i\sqrt{p-2})$ and $\QQ(i\sqrt{(p+2)(p-2)})$ have ``small'' discriminants for infinitely many primes $p$. In the spirit of Brauer-Siegel estimates, it proves that the class numbers of these imaginary quadratic fields are relatively prime to $p$, and so prove their $p$-rationality.
\end{abstract}
\small

\textbf{Keywords} : number theory, $p$-rational fields, Greenberg's conjecture, primes in arithmetic progression, multiquadratic number fields.

\textbf{Aknowledment} : Our thanks goes to Daniel Fiorilli, who provided the initial idea to prove our analytic proposition. Supported by ANR Flair (\textsc{ANR-17-CE40-0012}) and Bourgogne-Franche-Comté (grant \textsc{Ga Crococo}).
\normalsize

\section{Introduction to $p$-Rational Number Fields}

In 2016, R. Greenberg described a method (in \cite{GRE16}) to construct Galois extensions of $\mathbb{Q}$ with Galois group isomorphic to an open subgroup of $GL_{n}(\mathbb{Z}_p)$ (for various values of $n$ and primes $p$). His method is based on the (conjectured) existence of \emph{$p$-rationnal fields} (as defined thereafter) of prescribed Galois group. More specifically, Greenberg conjectured that, for any odd prime $p$ and any natural integer $t$, there exists a $p$-rational number field $K$ whose Galois group over $\mathbb{Q}$ is $(\mathbb{Z} / 2 \mathbb{Z})^t$ (\cite{GRE16} Conjecture 4.8).
~\\

We begin by giving the reader a quick overview of the notion of $p$-rational fields. 

Let $K$ be a number field (i.e. a finite extension of $\mathbb{Q}$), and $p$ a prime number. In the following, the various mathematical objects being defined will depend on $K$ and $p$, even if it does not appear in the notation most of the time.

An extension $L/K$ is said to be ``$p$-ramified'' if the extension is unramified outside the places above $p$. Let $M$ be the maximal $p$-ramified pro-$p$-extension of $K$, and $G = \mathrm{Gal}(M/K)$. We also denote by $M^{\mathrm{ab}}$ the maximal abelian $p$-ramified pro-$p$-extension of $K$, and $G^{\mathrm{ab}} = \mathrm{Gal}(M^{\mathrm{ab}} / K)$. One can notice that $G^{\mathrm{ab}}$ is the abelianization of $G$, and $M^{\mathrm{ab}}$ is the subfield of $M$ fixed by $G^{\mathrm{ab}}$. We also set $\widetilde{K}$ to denote the compositum of all $\ZZ_p$-extensions of $K$ (i.e. extensions of $K$ whose Galois group over $K$ is isomorphic to the additive group $\ZZ_p$) : as every $\ZZ_p$-extension is abelian and $p$-ramified, $\widetilde{K}$ is contained in $M^{\mathrm{ab}}$.

It is known, by class field theory, that $G^{\mathrm{ab}}$ is a $\mathbb{Z}_p$-module of rank $1 + r_2 + \delta$ (see \cite{MOV88} Chap II. \S 1 for example), where $\delta \geq 0$ is the defect in $p$ of the Leopoldt's conjecture for the number field $K$, and $r_2$ is the number of pairs of complex embeddings of $K$. Thus, there is an isomorphism of $\mathbb{Z}_p$-modules $G^{\mathrm{ab}} \simeq \mathbb{Z}_{p}^{1 + r_2 + \delta} \times \mathfrak{X}$,  
where $\mathfrak{X}$ is the $\mathbb{Z}_p$-torsion sub-group of $G^{\mathrm{ab}}$. Using $\widetilde{K}$, we have : $\mathrm{Gal}(\widetilde{K}/K) \simeq \ZZ_{p}^{1 +r_2 + \delta}$ and $\mathrm{Gal}(M^{\mathrm{ab}} / \widetilde{K}) \simeq \mathfrak{X}$.

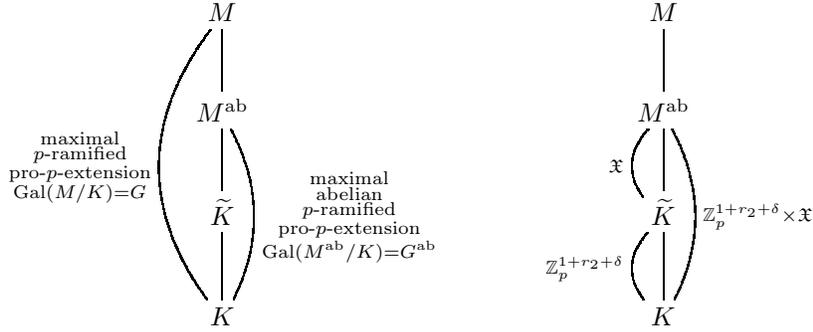
\begin{figure}[h!]
\centering
\begin{subfigure}[b]{0.3\textwidth}
\[
\xymatrix{
M \\
M^{\mathrm{ab}} \ar@{-}[u] \\
\widetilde{K} \ar@{-}[u] \\
K \ar@{-}[u] 
\ar@/^2pc/@{-}[uuu]^{\substack{\text{maximal} \\ p\text{-ramified} \\ \text{pro-}p\text{-extension} \\ \mathrm{Gal}(M/K) = G}} 
\ar@/_1pc/@{-}[uu]_{\substack{\text{maximal} \\ \text{abelian} \\ p\text{-ramified} \\ \text{pro-}p\text{-extension} \\ \mathrm{Gal}(M^{\mathrm{ab}} /K) = G^{\mathrm{ab}} }} 
}
\]
\end{subfigure}
\hspace{2.5cm}
\begin{subfigure}[b]{0.3\textwidth}
\[
\xymatrix{
M \\
M^{\mathrm{ab}} \ar@{-}[u] \\
\widetilde{K} \ar@{-}[u] \ar@/^1pc/@{-}[u]^{\mathfrak{X}} \\
K  \ar@{-}[u]
\ar@/^1pc/@{-}[u]^{\ZZ_{p}^{1 + r_2 + \delta}}
\ar@/_1pc/@{-}[uu]_{\ZZ_{p}^{1 + r_2 + \delta} \times \mathfrak{X} }
}
\]
\end{subfigure}
\end{figure}

With the notations above, we have :

\begin{theorem}[\cite{MOV88} Prop. 1] \label{theodefi}
The following conditions are equivalent :
\begin{enumerate}
\item $K$ satisfies the Leopoldt's Conjecture in $p$ (i.e. $\delta = 0$) and $\mathfrak{X}$ is trivial.
\item $G^{\mathrm{ab}} = \mathrm{Gal}(M^{\mathrm{ab}} / K) \simeq \mathbb{Z}_{p}^{1 + r_2}$, i.e. $G^{\mathrm{ab}}$ is a free $\ZZ_p$-module of rank $1 + r_2$.
\item $G = \mathrm{Gal}(M / K)$ is a free pro-$p$-group with $1 + r_2$ generators.
\end{enumerate}
\end{theorem}

\begin{definition}[$p$-Rational Field] \label{defiprat} A number field $K$ is said to be ``$p$-rational'' if it satisfies any (and therefore all) of the conditions of Theorem~\ref{theodefi}.
\end{definition}

We can give some fundamental examples of $p$-rational fields :
\begin{enumerate}
\item The field of rational numbers $\QQ$ is $p$-rational for every prime $p$.
\item The cyclotomic field $\QQ(\zeta_{p^n})$ is $p$-rational for $p$ a regular prime (i.e. $p$ a prime such that $p$ does not divide the class number of $\QQ(\zeta_p)$).
\item An imaginary quadratic field $K$ is $p$-rational for all (but a finite number of) primes $p$ (see \S 2).
\end{enumerate}

\begin{remark}
In the most general case, G. Gras conjectured in \cite{GRA16} that any number field $K$ is $p$-rational for all but a finite number of primes $p$.
\end{remark}

\begin{remark}
There exists algorithmic methods to determine if a given number field $K$ is $p$-rational for a given prime $p$, based on underlying profound algebraic results : see \cite{GRA19} for some useful \textsc{Pari/Gp} algorithms.
\end{remark}

We will use the following criterion of $p$-rationality for abelian number fields :

\begin{proposition}[\cite{GRE16} Proposition 3.6] \label{critsub}
Let $K$ be an abelian number field. Suppose that $[K : \QQ]$ is not divisible by $p$. Then $K$ is $p$-rational if and only if every cyclic extension of $\QQ$ contained in $K$ is $p$-rational.
\end{proposition}

\begin{remark} A special case, which will be of constant use throughout this paper, is the following : if $K$ is a multiquadratic number field (i.e. $\mathrm{Gal}(K/\QQ) \simeq (\ZZ / 2 \ZZ)^{t}$) and $p \neq 2$ is prime, then $K$ is $p$-rational if and only if all quadratic subfields of $K$ are $p$-rational.

In the most simple case, let $\QQ(\sqrt{n})$ and $\QQ(\sqrt{m})$ be two distinct quadratic number fields and let $p \neq 2$ a prime number. As stated previously, the biquadratic field $\QQ(\sqrt{n}, \sqrt{m})$ is $p$-rational if and only if the quadratic subfields $\QQ(\sqrt{n})$, $\QQ(\sqrt{m})$ and $\QQ(\sqrt{nm})$ are all $p$-rational. We can notice that it is not enough for $\QQ(\sqrt{n})$ and $\QQ(\sqrt{m})$ to be $p$-rational in order for their compositum $\QQ(\sqrt{n}, \sqrt{m})$ to be $p$-rational as well. For example, the quadratic fields $\QQ(\sqrt{2})$ and $\QQ(\sqrt{19})$ are $5$-rational. Nonetheless, their compositum $\QQ(\sqrt{2},\sqrt{19})$ is not $5$-rational, because its quadratic subfield $\QQ(\sqrt{38})$ is not. (One can test the $5$-rationality of these fields using the \textsc{Pari/Gp} algorithm of \cite{GRA19}).
\end{remark}

The notion of $p$-rational fields has been investigated by many authors since (at least) the 1980's. They were used -- at first -- to exhibit non-abelian number fields satisfying the Leopoldt's conjecture.

In 2016, Greenberg revisited the notion in \cite{GRE16} and conjectured the following :

\begin{conjecture}[\cite{GRE16}, Greenberg, 2016] \label{conjgre} ~\\ For any odd prime $p$ and any natural integer $t$, there exists a $p$-rational number field $K$ whose Galois group over $\mathbb{Q}$ is $(\mathbb{Z} / 2 \mathbb{Z})^t$.
\end{conjecture}

This conjecture, and in particular the cases of quadratic and biquadratic fields, has been lately invastigated by many authors, such as Barbulescu and Ray (\cite{BR}, 2019), Gras (\cite{GRA19}, 2019), Assim and Bouazzaoui (\cite{AB19}, 2020), Benmerieme and Movahhedi (\cite{BM20}, 2021). They proved this conjecture for quadratic et biquadratic fields and all odd primes $p$ (the cases $t=1$ and $t=2$). 

\section{Quadratic and Biquadratic $p$-Rational Number Fields}

\subsection{Quadratic $p$-rational number fiels}

\subsubsection{Criteria of $p$-rationality for quadratic numbers fields}

For the special (and elementary) case of quadratic fields, one has special criteria of $p$-rationality, such as the following :

\begin{proposition}[\cite{GRE16} Prop. 4.1 or \cite{BM20} Coro. 2.6] \label{critqua}
Suppose that $K$ is a quadratic field and that either $p \geq 5$ or that $p = 3$ and is unramified in $K / \QQ$.
\begin{enumerate}
\item Assume that $K$ is imaginary. Then $K$ is $p$-rational if and only if the $p$-Hilbert class field of $K$ is contained in the anti-cyclotomic $\ZZ_p$-extension of $K$. In particular, if $h_K$ is not divisible by $p$, then $K$ is $p$-rational.
\item Assume that $K$ is real. Let $\varepsilon_0$ be the fundamental unit of $K$. Then $K$ is $p$-rational if and only if $p \nmid h_K$ and $\varepsilon_0$ is not a $p$-th power in the completion $K_v$, for at least a place $v \mid p$.
\end{enumerate}
\end{proposition}

\begin{remark} The assumption on $p$ only guarantees that $\mu_p \not\subset K_v$ for each place $v$ of $K$ dividing $p$. The previous proposition is proved in two different ways in \cite{GRE16} and \cite{BM20}. In \cite{BM20}, the authors also give alternative criteria of $p$-rationality for real quadratic fields.
\end{remark}

\subsubsection{Imaginary Case}

For $p$ an odd prime, we know the condition $p \nmid h_K$ is a sufficient condition of $p$-rationality when $K$ is quadratic imaginary and $p \geq 5$ (or $p=3$ and $p$ not ramified in $K/\mathbb{Q}$) (Propositon~\ref{critqua}). We can state important corollaries : 

\begin{corollary}
Let $K$ be an imaginary quadratic number field. Then the set of primes $p$ for which $K$ is not $p$-rational is finite, and is contained in the set of divisors of $h_K$. In particular, an imaginary quadratic number field is $p$-rational for all but a finite number of primes $p$.
\end{corollary}

\begin{corollary}
For every odd prime $p$, the number field $\QQ(i)$ is $p$-rational.
\end{corollary}

In general, to prove $p \nmid h_K$ for a given imaginary quadratic field, we use :
\begin{lemma}[\cite{LOU05} \S 1, Prop.2] \label{lemmlou}
If L is an imaginary quadratic field and $-d$ is its discriminant (with $d > 0$), we have
\[ h_L \leq \frac{w_L \cdot \sqrt{d_L}}{4\pi} (\log(d_L) + 1+ \gamma - \log(\pi)) \]
with
\[ w_L = \begin{cases}
6 & \text{if } d = 3 \\
4 & \text{if } d = 4 \\
2 & \text{if } d \geq 5
\end{cases} \]
and $\gamma$ the Euler constant. Moreover, we have $1+ \gamma - \log(\pi) \leq \frac{1}{2}$.
\end{lemma}

\begin{remark} 
We can also show that, for every odd prime $p$, there even exists \emph{infinitely many} $p$-rationnal imaginary quadratic fields. To prove it, we can use a result by Hartung (\cite{HAR74}) proving there exists infinitely many imaginary quadratic fields $K$ such that $p$ does not divide the class number $h_K$. 

The reader may notice Hartung proves this theorem only in the case $p=3$. But, as stated in the article itself, the method can be easily adapted to prove the theorem for every odd prime $p$.
\end{remark}

\subsubsection{Real Case}

The real case is much more complicated : units of these real  quadratic fields play their part in the complexity of the setting.

\begin{example}
Some examples of real quadratic $p$-rational fields may even be found in the litterature prior to Greenberg's statement, for example in \cite{BYE01}, where the author does not use the terminology of $p$-rationality. However, for every prime $p \geq 5$, \cite{BYE01} Prop. 3.1 exhibits a real quadratic number field satisfying conditions which are sufficient to assert the $p$-rationality. The author even proves that there exists infinitely many real quadratic fields verifying these conditions, and such fields are $p$-rational. Thus, for every prime $p \geq 5$, there exists infinitely many $p$-rational real quadratic fields.
\end{example}

Many examples of $p$-rational real quadratic fields may be found in the litterature. In particular, we will use the following real quadratic fields :

\begin{proposition}[\cite{BM20} Prop. 4.4] \label{propmova}
For all prime $p \geq 5$, the real quadratic fields $\QQ(\sqrt{p(p-2)})$, $\QQ(\sqrt{p(p+2)})$ and $\QQ(\sqrt{(p-2)(p+2)})$ are $p$-rational.
\end{proposition}

\begin{remark} Notice that the proof is made possible because, for all these fields, we have a simple explicit formula for their fundamental unit. We will use the $p$-rationality of these fields to prove our main theorem. \end{remark}

\subsection{Biquadratic $p$-Rational Number Fields}

Using known $p$-rational quadratic number fields, it is possible to construct biquadratic $p$-rational number fields (under certain conditions).

In 2020, Benmerieme and Movahhedi proved in \cite{BM20} the following :

\begin{proposition}[\cite{BM20} Prop. 4.4] 
For all prime $p \geq 5$, the real biquadratic field \[ \QQ(\sqrt{p(p-2)}, \sqrt{p(p+2)}) \] is $p$-rational.
\end{proposition}

\begin{remark} In \cite{BM20}, the authors prove this proposition by considering every quadratic subfield of $\QQ(\sqrt{p(p-2)}, \sqrt{p(p+2)})$ : specifically, they prove $\QQ(\sqrt{p(p+2)})$, $\QQ(\sqrt{p(p-2)})$ and $\QQ(\sqrt{(p+2)(p-2)})$ are $p$-rational for all prime $p \geq 5$ (see Proposition \ref{propmova}).
\end{remark}

\begin{remark} In a similar way, we could exhibit other $p$-rational biquadratic fields, for $p \geq 5$, such as $\QQ(\sqrt{-p},\sqrt{-(p+2)})$ (\cite{BM20} Prop. 4.2), or $\QQ(\sqrt{-(p-1)}, \sqrt{-(p+1)})$ for $p \equiv 3 \pmod{4}$ (using the $p$-rationality of $\QQ(\sqrt{p^2 -1})$ as proved in this case by Byeon in \cite{BYE01}).
\end{remark}

\section{Triquadratic Rational Number Fields}

We state our main result :
\begin{theorem}
There exists infinitely many primes $p$ such that \[ \QQ(\sqrt{p(p+2)}, \sqrt{p(p-2)},i) \] is $p$-rational.
\end{theorem}

\begin{proof}
To prove that $K = \QQ(\sqrt{p(p+2)}, \sqrt{p(p-2)},i)$ is $p$-rational for infinitely many primes $p$, we will prove there exists infinitely many primes $p$ such that the following quadratic subfields are all $p$-rationnal :
\begin{multicols}{2}
\begin{enumerate}
\item $K_1 = \QQ(\sqrt{p(p+2)})$
\item $K_2 = \QQ(\sqrt{p(p-2)})$
\item $K_3 = \QQ(\sqrt{(p+2)(p-2)})$
\item $K_4 = \QQ(i)$
\item $K_5 = \QQ(i\sqrt{p(p+2)})$
\item $K_6 = \QQ(i\sqrt{p(p-2)})$
\item $K_7 = \QQ(i\sqrt{(p+2)(p-2)})$.
\end{enumerate}
\end{multicols}

Indeed, we already know $K_4$ is $p$-rationnal for all odd prime $p$. We also know $K_1$, $K_2$ et $K_3$ are $p$-rational for all prime $p \geq 5$ , according to Proposition \ref{propmova}.

Then, we use the analytic proposition proved apart in the next section (Proposition~\ref{propan})~: for $A > 0$, there exists infinitely many primes $p$ such that $p-2$ has a square factor larger than $(\log p)^A$ and $p+2$ has a square factor larger than $(\log p)^A$ as well. In particular, we can set $A = 2$, and we get, for all such $p$~:
\begin{align*}
\mathrm{disc}(K_5) & \leq \frac{4p(p+2)}{(\log p)^{4}} \\
\mathrm{disc}(K_6) & \leq \frac{4p(p-2)}{(\log p)^{4}} \\
\mathrm{disc}(K_7) & \leq \frac{4(p+2)(p-2)}{(\log p)^{8}}
\end{align*}

From now on, we'll only consider primes $p \geq 5$ lying in the infinite set of primes described previously, which we will denote by $\mathcal{P}$. 

Using Lemma~\ref{lemmlou}, for every $p \in \mathcal{P}$, we get :

\allowdisplaybreaks
\begin{align*}
h(K_5) & \leq \frac{6}{4\pi} \sqrt{ \frac{4p(p+2)}{(\log p)^4} } \left( \log \left( \frac{4p(p+2)}{(\log p)^4} \right) + \frac{1}{2} \right) \ll  \frac{p}{\log(p)} \\
h(K_6) & \leq \frac{6}{4\pi} \sqrt{ \frac{4p(p-2)}{(\log p)^4} } \left( \log \left( \frac{4p(p-2)}{(\log p)^4} \right) + \frac{1}{2} \right) \ll \frac{p}{\log(p)} \\
h(K_7) & \leq \frac{6}{4\pi} \sqrt{ \frac{4(p+2)(p+2)}{(\log p)^8} } \left( \log \left( \frac{4(p+2)(p+2)}{(\log p)^8} \right) + \frac{1}{2} \right) \ll \frac{p}{(\log p)^3}
\end{align*}

with absolute and effective implied constants.

Then, as $p \in \mathcal{P}$ tends to infinity, $p$ gets larger than $h(K_5)$, $h(K_6)$ and $h(K_7)$.

Thus, if $p$ is a big enough prime in $\mathcal{P}$, then $p \nmid h(K_5)$, $p \nmid h(K_6)$ and $p \nmid h(K_7)$, and these fields are $p$-rational.

In conclusion, there exists infinitely many primes $p$ such that all quadratic subfields of $K = \QQ(\sqrt{p(p+2)},\sqrt{p(p-2)},i)$ are $p$-rational, so $K$ itself is $p$-rational for infinitely many  primes $p$.

\end{proof}

\section{Proof of the Analytic Proposition}

Now, we are going to prove the analytic proposition we previously used to prove our main theorem :

\begin{proposition} \label{propan} Let $A > 0$. 

There exists infinitely many primes $p$ such that : there exists $m,n \in \mathbb{N}$ with
\[ \begin{cases} n^2 \mid p-2 \\
m^2 \mid p+2 \\
(\log p)^A < n \\
(\log p)^A < m 
\end{cases} \]
In other words, there exists infinitely many primes $p$ such that : $p-2$ has a square factor larger than $(\log p)^A$, and $p+2$ has a square factor larger than $(\log p)^A$.
\end{proposition}

Notations : for positive integers $m,n$, we set
\begin{align*} 
G(m) & : = \left\{ p \text{ prime} \left| \begin{array}{l}  
p \equiv -2 \pmod{m^2} \\ 
(\log p)^A < m \\ \end{array} \right. \right\} \\ 
H(n) & := \left\{ p \text{ prime} \left| \begin{array}{l}  
p \equiv 2 \pmod{n^2} \\ 
(\log p)^A < n \\ \end{array} \right. \right\}
\end{align*}
and
\[ I(m,n) := G(m) \cap H(n) = \left\{ p \text{ prime} \left| \begin{array}{l} p \equiv - 2 \pmod{m^2} \\ 
p \equiv 2 \pmod{n^2} \\ 
(\log p)^A < m\\
(\log p)^A < n \\ \end{array} \right. \right\} \] 

In order to prove the previous proposition, it is sufficient to show that the sum
\[ \sum_{\substack{p \text{ prime} \\ 3 \leq p < x}} \log(p) \left( \sum_{\substack{m \in \mathbb{N} \\  m < \sqrt{x+2} \\ p \in G(m) }} 1 \right) \times \left( \sum_{\substack{n \in \mathbb{N} \\ n < \sqrt{x-2} \\ p \in H(n) }} 1 \right)
\]
tends to infinity as $x \rightarrow + \infty$.

\subsection{First lower bound}

\allowdisplaybreaks

First, we have
\begin{align}
&  \sum_{\substack{p \text{ prime} \\ 3 \leq p < x}} \log(p) \left( \sum_{\substack{m \in \mathbb{N} \\ m < \sqrt{x+2} \\ p \in G(m) }} 1 \right) \times \left( \sum_{\substack{n \in \mathbb{N} \\ n < \sqrt{x-2} \\ p \in H(n) }} 1 \right) \label{somme1}  \\ 
& = \sum_{\substack{p \text{ prime} \\ 3 \leq p < x}} \sum_{\substack{m,n \in \mathbb{N} \\ m < \sqrt{x+2} \\ n < \sqrt{x-2} \\ p \in I(m,n)}} \log(p) \\
& = \sum_{\substack{m,n \in \mathbb{N} \\ m < \sqrt{x+2} \\ n < \sqrt{x-2} }} \sum_{\substack{p \text{ prime} \\ 3 \leq p < x \\ p\in I(m,n)}} (\log p) \\
& \geq \sum_{\substack{m,n \in \mathbb{N} \\ m < \sqrt{x+2} \\ n < \sqrt{x-2} \\ (m,n) = 1 \\ m,n \text{ odd}}} \sum_{\substack{p \text{ prime} \\ 3 \leq p < x \\ p \in I(m,n)}} \log(p) \label{somme2}
\end{align}

We obtain a first lower bound by restricting the first sum (on $m,n \in \mathbb{N}$) with the conditions $(m,n)=1$ and $m,n$ odd. We add the first condition in order to use the Chinese remainder theorem : as $n$ and $m$ are coprime, so are $m^2$ and $n^2$, and there exists $a_{m,n} \in \mathbb{Z}$ such that, for all $k \in \mathbb{Z}$, we have :
\[ \begin{cases} k \equiv 2 & [n^2] \\ k \equiv -2 & [m^2] \end{cases} \quad \iff \quad k \equiv a_{m,n} [m^2 n^2] \]
Then, if $m$ and $n$ are coprime, we get
\[ I(m,n) = \left\{ p \text{ prime} \left| \begin{array}{l} p \equiv a_{m,n} \pmod{m^2 n^2} \\ (\log p)^A < m \\ (\log p)^A < n \end{array} \right. \right\} \]
Moreover, the condition ``$m,n$ odd'' ensures that $a_{m,n}$ and $m^2 n^2$ are coprime.

~\\
Let $B>0$ such that $A <B$. For $x$ big enough (in particular, such that $\sqrt{x-2}$ is larger than $(\log x)^B$), we may bound~\eqref{somme2} from below by
\begin{equation} \sum_{\substack{ (\log x)^A < m,n < (\log x)^B \\ (m,n)=1 \\ m,n \text{ odd}}} \sum_{\substack{ 3 \leq p < x \\ p \in I(m,n) }} \log(p)  \label{somme3}
\end{equation}

We notice that :
\[ \begin{cases} p < x \\ (\log p)^A < n \\ (\log p)^A < m \end{cases} \iff \begin{cases} (\log p)^A < (\log x)^A \\ (\log p)^A < n \\ (\log p)^A < m \end{cases}  \iff (\log p)^A < \min\left( (\log x)^A, n, m \right) \]
Then, if $(\log x)^A < m,n$ (as in~\eqref{somme3}), we get $\min((\log x)^A,n,m) =  (\log(x))^A$, and the previous conditions are equivalent to $(\log p)^A < (\log x)^A$, i.e. $p < x$.

Thus~\eqref{somme3} is equal to
\begin{equation} \label{somme4}
\sum_{\substack{ (\log x)^A < m,n < (\log x)^B \\ (m,n)=1\\ m,n \text{ odd}}} \sum_{\substack{ 3 \leq p < x \\ p \equiv a_{m,n} [m^2 n^2] }} \log(p)
\end{equation}
which is a convenient lower bound of our first sum~\eqref{somme1}.

\subsection{Siegel-Walfisz's Theorem}

\begin{definition}
For $(a,q)=1$, we define

\[ \theta(x;q,a) = \sum_{\substack{ p \text{ prime} \\ p < x \\ p \equiv a [q] }} \log(p) \qquad \text{et} \qquad \psi(x;a,q) = \sum_{\substack{ n < x \\ n \equiv a [q] }} \Lambda(n) \]

with $\Lambda(n) = \begin{cases} \log(p) & \text{if } n = p^k \text{ with } p \text{ prime} \\
0 & \text{else} \end{cases}$
\end{definition}

The prime number theorem gives an equivalent at infinity for $\theta$ and $\psi$. An estimation of the error term is given by the following theorem :

\begin{theorem}[\textsc{Siegel-Walfisz}, \cite{IWAKOW} Corollary 5.29] Let $a,q \in \mathbb{N}$, $q \geq 1$, $(a,q) = 1$. Let $C > 0$. For $q \ll \log(x)^C$, the following holds
\[ \psi(x;q,a) = \frac{x}{\varphi(q)} + O\left( \frac{x}{(\log x)^C} \right) \]
for every $x \geq 2$. 
The implied constant depends only on $C$.
\end{theorem}

Later in the proof, we will need this result using $\theta$ instead of $\psi$. 
First, we check that $\displaystyle \psi(x;q,a) = \theta(x;q,a) + O \left( x^\frac{1}{2} \log(x)^2 \right)$ (see \cite{HW} XXII. Theorem 413 p.452)

Then, the previous theorem may be adapted in the following way :

\begin{corollary} \label{coroSW} Let $a,q \in \mathbb{N}$, $q \geq 1$, $(a,q) = 1$. 

Let $C > 0$. For $q \ll \log(x)^C$, the following holds
\[ \theta(x;q,a) = \frac{x}{\varphi(q)} + O\left( \frac{x}{(\log x)^C} \right) \]
for every $x \geq 2$. The implied constant depends only on $C$.
\end{corollary}

\subsection{Equivalent at infinity}

We want to prove sum~\eqref{somme1} tends to infinity as $x \rightarrow + \infty$. We bounded~\eqref{somme1} from below by~\eqref{somme4}, which again can be bounded from below (thanks to the previous notations) by :

\begin{equation}
\sum_{\substack{ (\log x)^A < m,n < (\log x)^B \\ (m,n)=1\\ m,n \text{ odd}}}  \left[ \theta \left( x , a_{m,n}, m^2 n^2 \right) - \log(2) \right]
\end{equation}

Remark : the definition of $\theta$ that we previously gave also involves $p=2$, while our initial sum did not, which explains the appearance of a $\log(2)$ term, which makes no difference whatsoever in the end.
~\\

By Corollary \ref{coroSW}, if we set $C > 4B$, we have $m^2 n^2 \leq (\log x)^{4B} < (\log x)^C$, and we get 

\begin{equation}
\sum_{\substack{ (\log x)^A < m,n < (\log x)^B \\ (m,n)=1\\ m,n \text{ odd}}}  \theta \left( x , a_{m,n}, m^2 n^2 \right) =
\sum_{\substack{ (\log x)^A < m,n < (\log x)^B \\ (m,n)=1\\ m,n \text{ odd}}}  \left[ \frac{x}{\varphi(m^2 n^2)} + E_{m,n}(x) \right] 
\end{equation}

with $\displaystyle E_{m,n}(x) \ll \frac{x}{(\log x)^C}$, the implied constant depending only on $C$.

We can notice that we are able to use these estimates as we chose $m,n$ odd, so that the residue $a_{m,n}$ and the modulus $m^2 n^2$ are coprime.

\newpage
Thus, we need to prove :

\begin{enumerate} \item[(i)] that the main term

\[ \sum_{\substack{ (\log x)^A < m,n < (\log x)^B \\ (m,n)=1\\ m,n \text{ odd}}} \frac{x}{\varphi(m^2 n^2)}  \]

tends tends to infinity as $x \rightarrow + \infty$, and 
\item[(ii)] that the growth of the main term is not impeached by the error term~:

\[ \sum_{\substack{ (\log x)^A < m,n < (\log x)^B \\ (m,n)=1\\ m,n \text{ odd}}} E_{m,n}(x) - \sum_{\substack{ (\log x)^A < m,n < (\log x)^B \\ (m,n)=1\\ m,n \text{ odd}}} \log(2) \]

\end{enumerate}

\subsection{Estimation of the Main Term}

We are going to prove that

\[ \sum_{\substack{ (\log x)^A < m,n < (\log x)^B \\ (m,n)=1\\ m,n \text{ odd}}} \frac{x}{\varphi(m^2 n^2)} \underset{x \rightarrow \infty}{\longrightarrow} \infty \]
by bounding this term below and give an equivalent at infinity of this lower bound. First, we use the inequality $\varphi(m^2 n^2) \leq m^2 n^2$, so that

\begin{equation}
\sum_{\substack{ (\log x)^A < m,n < (\log x)^B \\ (m,n)=1\\ m,n \text{ odd}}} \frac{x}{\varphi(m^2 n^2)} \geq \sum_{\substack{ (\log x)^A < m,n < (\log x)^B \\ (m,n)=1\\ m,n \text{ odd}}} \frac{x}{m^2 n^2}
\end{equation}

Then, we get

\begin{align*}
&  \displaystyle \sum_{\substack{ (\log x)^A < m,n < (\log x)^B \\ (m,n)=1\\ m,n \text{ odd}}} \frac{1}{m^2 n^2} \\ 
 = &  \displaystyle \sum_{\substack{ (\log x)^A < (2m+1),(2n+1) < (\log x)^B \\ ((2m+1),(2n+1))=1}} \frac{1}{(2m+1)^2 (2n+1)^2} \\ 
 = & \displaystyle \sum_{ (\log x)^A < (2m+1), (2n+1) < (\log x)^B } \frac{1}{(2m+1)^2 (2n+1)^2} \\
& \qquad - \underbrace{ \sum_{\substack{ (\log x)^A < (2m+1),(2n+1) < (\log x)^B \\ ((2m+1),(2n+1))> 1}} \frac{1}{(2m+1)^2 (2n+1)^2} }_{(*)} \\
\geq & \displaystyle \sum_{ \frac{(\log x)^A - 1}{2} < m,n < \frac{(\log x)^B - 1}{2} } \frac{1}{(2m+1)^2 (2n+1)^2} \\
& \qquad - \underbrace{\sum_{2 \leq d \leq (\log x)^B} \sum_{\substack{ (\log x)^A < (2m+1), (2n+1) < (\log x)^B \\ d \mid 2m+1 \\ d \mid 2n + 1}} \frac{1}{(2m+1)^2 (2n+1)^2}}_{(**)} \\
\geq & \displaystyle \sum_{ \frac{(\log x)^A - 1}{2} < m,n < \frac{(\log x)^B - 1}{2} } \frac{1}{(2m+1)^2 (2n+1)^2} \\
& \qquad - \underbrace{ \sum_{2 \leq d \leq (\log x)^B} \sum_{\substack{ k,k' \text{ odd} \\ (\log x)^A < dk, dk < (\log x)^B }} \frac{1}{(dk)^2 (dk')^2} }_{(***)} \\
\end{align*}

Note that we have $(*) \leq (**) \leq (***)$, which explains the inequalities.

Thus, the latter is equal to 
\[ \left( \sum_{ \frac{(\log x)^A - 1}{2} < n < \frac{(\log x)^B - 1}{2} } \frac{1}{ (2n+1)^2} \right)^2  -  \sum_{2 \leq d \leq (\log x)^B} \frac{1}{d^4} \left( \sum_{\substack{ \frac{(\log x)^A - d}{2d} < n < \frac{(\log x)^B - d}{2d} }} \frac{1}{(2n+1)^2} \right)^2 \]

We now want to bound the last quantity from below. To do this, one uses lower and upper bounds which we get by comparisons beetwen series and integrals.

Specifically, we use the inequality
$\displaystyle \sum_{n=a}^{b} \frac{1}{(2n+1)^2} \geq \int_{a}^{b+1} \frac{\mathrm{d}t}{(2t+1)^2}$

so that we get 
\[ \sum_{ \frac{(\log x)^A -1}{2} < n < \frac{(\log x)^B - 1}{2} } \frac{1}{(2n+1)^2}
\geq \frac{1}{2 \left( (\log x)^A + 2 \right)} - \frac{1}{2 (\log x)^B}
 \]

In the same way, we use the inequality $\displaystyle \sum_{n=a}^{b} \frac{1}{(2n+1)^2} \leq \int_{a}^{b}  \frac{\mathrm{d}t}{(2t+1)^2} + \frac{1}{(2a+1)^2}$

to get 
\[
\displaystyle \sum_{\substack{ \frac{(\log x)^A - d}{2d} < n < \frac{(\log x)^B - d}{2d} }} \frac{1}{(2n+1)^2} \leq \displaystyle \frac{d}{2(\log x)^A } - \frac{d}{2(\log x)^B} + \frac{d^2}{(\log x)^{2A}}
 \]

Using the previous inequalities, we have

\allowdisplaybreaks

\begin{align*} &  \displaystyle \sum_{\substack{ (\log x)^A < m,n < (\log x)^B \\ (m,n)=1\\ m,n \text{ odd}}} \frac{1}{m^2 n^2} \\ 
\geq\;& \displaystyle \left( \sum_{ \frac{(\log x)^A - 1}{2} < n < \frac{(\log x)^B - 1}{2} } \frac{1}{ (2n+1)^2} \right)^2 \\
& \qquad -  \sum_{2 \leq d \leq (\log x)^B} \frac{1}{d^4} \left( \sum_{\substack{ \frac{(\log x)^A - d}{2d} < n < \frac{(\log x)^B-d}{2d} }} \frac{1}{n^2} \right)^2 \\ 
\geq\; & \displaystyle \left( \frac{1}{2 \left( (\log x)^A + 2 \right)} - \frac{1}{2 (\log x)^B} \right)^2 \\
& \qquad - \sum_{2 \leq d \leq (\log x)^B} \frac{1}{d^4} \left( \frac{d}{2(\log x)^A } - \frac{d}{2(\log x)^B} + \frac{d^2}{(\log x)^{2A}} \right)^2 \\
\geq\;& \displaystyle \frac{1}{4 \left( (\log x)^A + 2 \right)^2 } + \frac{1}{4 (\log x)^{2B} } - \frac{1}{2 \left( (\log x)^A + 2 \right) (\log x)^B } \\
& \displaystyle - \sum_{2 \leq d \leq (\log x)^B} \frac{1}{d^2} \left( \frac{1}{4(\log x)^{2A} } + \frac{1}{4(\log x)^{2B}} + \frac{d^2}{(\log x)^{4A}} + \frac{d}{(\log x)^{3A}} \right) \\
\geq\;& \displaystyle \frac{1}{4 \left( (\log x)^A + 2 \right)^2 } + \frac{1}{4 (\log x)^{2B} } - \frac{1}{2 \left( (\log x)^A + 2 \right) (\log x)^B } \\
& \displaystyle - \left( \sum_{2 \leq d \leq (\log x)^B} \frac{1}{d^2} \right) \left( \frac{1}{4(\log x)^{2A} } + \frac{1}{4(\log x)^{2B}} \right) \\
& \displaystyle - \left( \sum_{2 \leq d \leq (\log x)^B} \frac{1}{d} \right) \frac{1}{(\log x)^{3A}}  - \left( \sum_{2 \leq d \leq (\log x)^B} 1 \right) \frac{1}{(\log x)^{4A}} \\
\geq\;& \displaystyle \frac{1}{4 \left( (\log x)^A + 2 \right)^2 } + \frac{1}{4 (\log x)^{2B} } - \frac{1}{2 \left( (\log x)^A + 2 \right) (\log x)^B } \\
& \displaystyle - (\zeta(2) - 1) \left( \frac{1}{4(\log x)^{2A} } + \frac{1}{4(\log x)^{2B}} \right) \\
& \displaystyle - (1 + \log(\log(x)^B) ) \frac{1}{(\log x)^{3A}}  - \frac{(\log x)^B}{(\log x)^{4A}} \\
\end{align*}

and, if we take $A < B < 2A$, the latter is finally equivalent to $\displaystyle \frac{2 - \zeta(2)}{4} \frac{1}{(\log x)^{2A}}$ 

Thus, we proved that

\[ \sum_{\substack{ (\log x)^A < m,n < (\log x)^B \\ (m,n)=1\\ m,n \text{ odd}}} \frac{1}{m^2 n^2}  \]
 is larger than a quantity which is equivalent at infinity to
\[ \frac{2 - \zeta(2)}{4} \frac{1}{(\log x)^{2A}} \]

Subsequently, the main term
\[ \sum_{\substack{ (\log x)^A < m,n < (\log x)^B \\ (m,n)=1\\ m,n \text{ impairs}}} \frac{x}{\varphi(m^2 n^2)}  \]
is larger than a quantity which is equivalent at infinity to
\begin{equation} 
\frac{2 - \zeta(2)}{4} \frac{x}{(\log x)^{2A}} \label{mino}
\end{equation}

Thus, the main term tends to infinity as $x \rightarrow + \infty$, at least as fast as~\eqref{mino}.

\subsection{Error Terms}

Remember that $E_{m,n}(x) = \left| \theta(x; a, m^2 n^2 ) - \frac{x}{\varphi(m^2 n^2)} \right|$. According to Siegel-Walfisz, for a given $C > 0$ and every $q \leq (\log x)^C$, the following holds
\[ \left| \theta(x; a, q ) - \frac{x}{\varphi(q)} \right| = O \left( \frac{x}{(\log x)^C} \right) \]

Then, if $4B < C$, for every big enough $x$, we know that $m^2 n^2 < (\log x)^C$ as soon as $(\log x)^A < m,n < (\log x)^B$. So each $E_{m,n}(x)$ is dominated by $\displaystyle \frac{x}{(\log x)^C}$ in the sum \[ \sum_{\substack{ (\log x)^A < m,n < (\log x)^B \\ (m,n)=1\\ m,n \text{ odd}}} E_{m,n}(x)\] 
Considering that we get at most $(\log x)^{2B}$ terms, we have~:
\[ \sum_{\substack{ (\log x)^A < m,n < (\log x)^B \\ (m,n)=1\\ m,n \text{ odd}}} E_{m,n}(x) \ll (\log x)^{2B} \frac{x}{(\log x)^C} \]

Noticing that we chose $C$ such that $C > 4B$, we have $C > 2A + 2B$, i.e. $2B - C < - 2A$ and we can conclude that

\[ \sum_{\substack{ (\log x)^A < m,n < (\log x)^B \\ (m,n)=1\\ m,n \text{ odd}}} E_{m,n}(x) \ll \frac{x}{(\log x)^{2A}} \]

The last part of the error term, namely \[ \sum_{\substack{ (\log x)^A < m,n < (\log x)^B \\ (m,n)=1\\ m,n \text{ odd}}} \log(2) \] is at most $(\log x)^{2B}$ times $\log(2)$, so it may be bounded above by $(\log x)^{2B} \log(2)$, which is also dominated by$ \frac{x}{(\log x)^{2A}}$.

\subsection{Conclusion}

We proved 
\[ \sum_{\substack{p \text{ prime} \\ 3 \leq p < x}} \log(p) \left( \sum_{\substack{m \in \mathbb{N} \\  m < \sqrt{x+2} \\ p \in G(m) }} 1 \right) \times \left( \sum_{\substack{n \in \mathbb{N} \\ n < \sqrt{x-2} \\ p \in H(n) }} 1 \right)
\]

was bounded below by
\[ \underbrace{\sum_{\substack{ (\log x)^A < m,n < (\log x)^B \\ (m,n)=1\\ m,n \text{ odd}}} \frac{x}{\varphi(m^2 n^2)} }_{A(x)}  + \underbrace{ \sum_{\substack{ (\log x)^A < m,n < (\log x)^B \\ (m,n)=1\\ m,n \text{ odd}}} E(x)}_{O(A(x))} + \underbrace{ \sum_{\substack{ (\log x)^A < m,n < (\log x)^B \\ (m,n)=1\\ m,n \text{ odd}}} \log(2) }_{O(A(x))}  \]
with $A(x) \rightarrow \infty$, showing that our first sum tends to infinity as $x \rightarrow \infty$. \quad $\square$

\section{Open questions}
\begin{enumerate}

\item In order to prove Greenberg's conjecture for every integer $t$ and every odd prime $p$, one could try at first to exhibit $p$-rational totally real triquadratic number fields. This would require to exhibit a well chosen set of real quadratic fields, whose fundamental units can be explicitally found.

\item It is possible, being given a $p$-rational totally real multiquadratic field $K$ of order $2^t$ for every odd prime $p$, that we could adapt our analytic proposition to prove $K(i)$ is $p$-rationnal for an infinity of prime $p$ (probably under certain restricting conditions).

\item The analytic proposition may be improved : for example, one may consider a finite collection of positive integers $m_i$ (and relatively prime residues $r_i$) and search for primes $p$ such that
\[ \forall i,\quad p \equiv r_i \pmod{m_{i}^{2}} \quad \text{and}\quad m_{i} > (\log p)^A \]
One may also give a proper estimation of the number of such $p$, rather than the lower bound we gave.

In appendix, we prove an alternate (stronger) version (under GRH) of the analytic proposition \ref{propan}.
\end{enumerate}

\begin{appendix}

\section{A stronger analytic proposition under GRH}

For $A > 0$, we proved there exists infinitely many primes $p$ such that $(p+2)$ and $(p-2)$ both admit square factors larger than $(\log x)^A$. We can prove, under GRH, a stronger result. More specifically :
\begin{proposition}
Let $\varepsilon < \frac{1}{8} $. Suppose GRH holds true for $L(s,\chi)$ with $\chi \pmod{q}$. There exists infinitely many primes $p$ for which there exists $m,n \in \mathbb{N}^{\ast}$, such that
\[ \begin{cases} 
n^2 \mid p-2 \\
m^2 \mid p+2 \\
p^{\varepsilon} < n \\
p^{\varepsilon} < m 
\end{cases} \]
\end{proposition}

\begin{proof} The proof is similar to the previous one. 

In the same way, we consider the sum
\[
\sum_{\substack{p \text{ prime} \\ 3 \leq p < x}} \log(p) \Bigg( \sum_{\substack{m \in \mathbb{N} \\ m  < \sqrt{x+2} \\ m^2 \mid p + 2 \\ p^\varepsilon < m }} 1 \Bigg) \times \Bigg( \sum_{\substack{ n \in \mathbb{N} \\ n < \sqrt{x-2} \\ n^2 \mid p -2 \\ p^\varepsilon < n }} 1 \Bigg)
\]

Then, using similar calculations as before, we show that the previous sum is bounded below by

\[
\sum_{\substack{m,n \in \mathbb{N} \\ m < \sqrt{x+2} \\ n < \sqrt{x-2} \\ (m,n)=1 \\ m,n \text{ odd}}} \sum_{\substack{ p \text{ prime} \\ 3 \leq p < x \\ p \equiv a_{m,n} [m^2 n^2] \\  p < m^{1/\varepsilon} \\ p < n^{1 / \varepsilon}}} \log(p)
\]
which is larger (for $x$ big enough) than
\begin{equation} \label{sommea}
\sum_{\substack{ x^\varepsilon < m,n < x^\alpha \\ (m,n)=1 \\ m,n \text{ odd}}}  \sum^{p \text{ prime}}_{\substack{p \equiv a_{m,n} [m^2n^2] \\ 3 \leq p < x}} \log(p)
\end{equation}
for $\varepsilon < \alpha < \frac{1}{2}$.
~\\

Then, under GRH for $L(s,\chi)$ with $\chi \pmod{q}$, we have (see \cite{IWAKOW} \S 17.1) :
\[ \theta(x; m^2 n^2, a_{m,n}) = \sum_{\substack{p \text{ prime} \\ p \equiv a_{m,n} [m^2n^2] \\ p < x}} \log(p) = \frac{x}{\varphi(m^2 n^2)} + O \left( x^{1/2} (\log x)^2 \right)
\]

So~\eqref{sommea} is greater than 
\[  \sum_{\substack{ x^\varepsilon < m,n < x^\alpha \\ (m,n)=1 \\ m,n \text{ odd}}} \left( \frac{x}{\varphi(m^2 n^2)} + E_{m,n} - \log(2) \right)
\]
with $E_{m,n} \ll x^{1/2} (\log x)^2$.
~\\

First, we show the main term
\[
\sum_{\substack{ x^\varepsilon < m,n < x^\alpha \\ (m,n)=1 \\ m,n \text{ odd}}} \frac{x}{\varphi(m^2 n^2)}  
\]
is larger than a quantity equivalent to $\displaystyle \frac{2 - \zeta(2)}{4} x^{1 - 2 \varepsilon}$ at infinity.

Indead, we have
\begin{align*}
\sum_{\substack{ x^\varepsilon < m,n < x^\alpha \\ (m,n)=1\\ m,n \text{ odd}}} \frac{1}{m^2 n^2} & = \left( \sum_{ \frac{x^\varepsilon - 1}{2} < n < \frac{x^\alpha - 1}{2} } \frac{1}{ (2n+1)^2} \right)^2 \\
& \quad -  \sum_{2 \leq d \leq x^\alpha} \frac{1}{d^4} \left( \sum_{\substack{ \frac{x^\varepsilon - d}{2d} < n < \frac{x^\alpha - d}{2d} }} \frac{1}{(2n+1)^2} \right)^2 \\
& \geq \left( \frac{1}{2(x^\varepsilon +2)} - \frac{1}{2x^\alpha} \right)^2 -  \sum_{2 \leq d \leq x^\alpha} \frac{1}{d^4} \left( \frac{d}{2(x^\varepsilon - 2d)} - \frac{d}{2x^\alpha} \right)^2 \\
& \geq \frac{1}{4(x^\varepsilon + 2)^2} - \frac{1}{2(x^\varepsilon +2)x^\alpha} + \frac{1}{4 x^{2\alpha}} \\
& \quad -  \frac{\zeta(2)-1}{4 (x^\varepsilon - 4)^2} +  \frac{\zeta(2)-1}{2(x^\varepsilon - 2x^\alpha)x^\alpha} - \frac{\zeta(2) -1}{4x^{2\alpha}}
\end{align*}

Therefore we can consider the first error term $\sum_{\substack{ x^\varepsilon < m,n < x^\alpha \\ (m,n)=1 \\ m,n \text{ odd}}} E_{m,n}(x)$ with $E_{m,n}(x) \ll x^{1/2} (\log x)^2$ (the implied constant being independent of $E_{m,n}$). As we get at most $x^{2\alpha}$ terms which are $O \left( x^{1/2} (\log x)^2 \right)$ (under GRH), this error term is $O \left( x^{\frac{1}{2} + 2 \alpha} (\log x)^2 \right)$.
 Thus, we take $\alpha$ such that $\varepsilon < \alpha < \frac{1}{4} - \varepsilon$ (justifying the fact that we took $\varepsilon < \frac{1}{8}$), so that $\frac{1}{2} + 2 \alpha < 1 - 2 \varepsilon$. Then, this first error term is $O \left( \frac{2 -\zeta(2)}{4} x^{1 - 2\varepsilon} \right)$.

Finally, the second error term, namely
$\sum_{\substack{ x^\varepsilon < m,n < x^\alpha \\ (m,n)=1 \\ m,n \text{ odd}}} (\log 2)$ 
is at most $x^{2 \alpha}$ times $\log(2)$, and may be bounded above by $x^{2 \alpha} \log(2) = O \left( \frac{2 - \zeta(2) }{4} x^{1 - 2 \varepsilon}  \right)$ (as $\alpha$ has been set previously such that $2 \alpha < \frac{1}{2} < 1 - 2\varepsilon$)

In conclusion, we bounded below our first sum by :

\[
\underbrace{\sum_{\substack{ x^\varepsilon < m,n < x^\alpha \\ (m,n)=1 \\ m,n \text{ odd}}} \frac{x}{\varphi(m^2 n^2)}}_{A(x)}  + \underbrace{\sum_{\substack{ x^\varepsilon < m,n < x^\alpha \\ (m,n)=1 \\ m,n \text{ odd}}} E_{m,n} + \sum_{\substack{ x^\varepsilon < m,n < x^\alpha \\ (m,n)=1 \\ m,n \text{ odd}}} \log(2)}_{O(A(x))}
\]

with $A(x) \rightarrow + \infty$, which is sufficient to prove the proposition.

\end{proof}

\end{appendix}

\end{document}